\documentclass[preprint,review,10pt]{amsart}
\usepackage{amsmath,amssymb,amsfonts,xspace}
\newtheorem{theorem}{Theorem}[section]

\theoremstyle{definition}
\newtheorem{definition}[theorem]{Definition}
\newtheorem{example}[theorem]{Example}

\newtheorem{proposition}[theorem]{Proposition}
\newtheorem{corollary}[theorem]{Corollary}

\theoremstyle{remark}
\newtheorem{remark}[theorem]{Remark}

\numberwithin{equation}{section}

\begin{document}

\title{ $Endo$-prime hyperideals }

\author{Mahdi Anbarloei}
\address{Department of Mathematics, Faculty of Sciences,
Imam Khomeini International University, Qazvin, Iran.
}

\email{m.anbarloei@sci.ikiu.ac.ir}





\begin{abstract}
In this paper, we aim to introduce and study the notion of   $Endo$-prime hyperideals.

\end{abstract}
\maketitle

\section{Introduction}
The notion of prime ideals  in the theory of rings as an extention
of the notion of prime number in the ring of integers has an important position in that theory, as might be expected from the central place occupied by the primes in arithmetic. For this reason, many papers were written abut prime ideals and thier generalizations.
Looking at the respective studies, the concept of $Endo$-prime ideals of  a commutative ring as an expansion of prime ideals has been introduced and studied in \cite{Akray} and \cite{Najjar }. Let $R$ be a commutative ring and $\theta: R \longrightarrow R$ be a fixed endomorphism. An ideal $I$ of $R$ refers to a $\theta$-prime ideal  if, whenever $u,v \in R$ and $uv \in I$, then $u \in I$ or $\theta(v) \in I$.  Moreover, a generalization of this notion called $\theta$-primary ideals was proposed in \cite{Mahdou}.

A hyperstructure  is an  algebraic structure  that has at least one multi-valued operation, known as a hyperoperation. It is well-known that hyperstructure theory was proposed  by a French mathematician in 1934 \cite{s1}. He defined the notion of hypergroups as an expansion of groups. This has been the starting point for the study of many types of hyperrings such as general hyperrings, Krasner hyperrings and multiplicative hyperrings. The Krasner hyperrings introduced by Krasner are obtained by considering the addition as a hyperoperation and the multiplication as a binary operation \cite{krasner}. Krasner $(m,n)$-hyperrings as an expansion of Krasner hyperrings were presented in \cite{d1}. More details about these hyperrings and fundamental relations on them were given in \cite{mah6,asadi, car, ma,  nour,  rev1}.  The important hyperideals of a Krasner $(m,n)$-hyperring, namely, $n$-ary prime and $n$-ary primary hyperideals were studies   in \cite{sorc1}. The generalizations of these notions such as $(k,n)$-absorbing, $(k,n)$-absorbing primary hyperideals, 
weakly   $(k,n)$-absorbing and weakly $(k,n)$-absorbing primary hyperideals were investigated in \cite{www} and \cite{rev2}.

In this paper, we aim to   analyze a class 
of hyperideals established on an endomorphism in a Krasner $(m,n)$-hyperring. The paper is organized as follows. In
Section 2, we start by recalling some background material. In
Section 3, the notion of $n$-ary $Endo$-prime hyperideals associated with an endomorphism $\theta$ is introduced and some of their properties are given. For example, we show that every $n$-ary prime hyperideal of a Krasner $(m,n)$-hyperring is an $n$-ary $Endo$-prime hyperideal but the converse need not to be hold in Example \ref{salami}. In Theorem \ref{1} it is shown that the radical of an $n$-ary $Endo$-prime hyperideal of $H$ is an $n$-ary $Endo$-prime hyperideal of $H$. We show that if  $E$ is an $n$-ary $Endo$-prime hyperideal  associated with $\theta$, then $E$ contains $\theta(E)$ in Proposition \ref{2}. But  Example \ref{pelarak} verifies that the converse of of the expression may not be true. Moreover, Theorem \ref{madar} describes the interaction of hyperideals with an $n$-ary $Endo$-prime hyperideal and provides conditions under which a hyperideal becomes an $n$-ary Endo-prime hyperideal.
In Theorem \ref{inter}, we investigate whether the intersection of the
collection of $n$-ary $Endo$-prime hyperideals preserves the algebraic structure. In Theorem \ref{5}, we obtain that  the intersection of all $n$-ary $Endo$-prime hyperideals associated with $\theta$ in a Krasner $(m,n)$-hyperring is equal to the set of all $\theta$-nilpotent elements of the Krasner $(m,n)$-hyperring where $\theta$ is an endomorphism. Besides, we examine the transfer of this new property of hyperideals to various ring-theoretic constructions. In Section 4, we devote our attention to the  study of a generalization of $n$-ary $Endo$-prime hyperideals in a Krasner $(m,n)$-hyperring called $n$-ary $Endo$-primary  hyperideals. 

\section{Preliminaries}
In this section, we peresent deﬁnitions and notations  used in this paper.

An $n$-ary hyperoperation $``h"$ on   set $H \neq \varnothing$ is a mapping of $H^n$ into the family of all non-empty subsets of $H$. If $``h"$ is an $n$-ary hyperoperation on $H$, then $(H, h)$ is called an $n$-ary hypergroupoid. We can expand the $n$-ary hyperoperation on $H$ to non-empty subsets of $H$ as follows. Let  $H_1,\ldots, H_n$ be subsets of $H$, then
\[h(H^n_1) = h(H_1,\ldots, H_n) = \bigcup \{h(u_1,\ldots,u_n) \ \vert \ u_i \in H_i, 1 \leq i \leq n \}.\]

Let us use the notation $u^j_i$ instead of   the sequence $u_i, u_{i+1},\ldots, u_j$. Then we get $h(u_1,\ldots, u_i, v_{i+1},\ldots, v_j, w_{j+1},\ldots, w_n)= h(u^i_1, v^j_{i+1},w^n_{j+1})$ and  $h(u_1,\ldots, u_i, \underbrace{v,\ldots,v}_{j-i},\break  w_{j+1},\ldots, w_n)=h(u^i_1, v^{(j-i)}, w^n_{j+1})$. This  notation is the  empty symbol when $j< i$. Let $h$ be an $n$-ary hyperoperation.  Then $r$-ary hyperoperation $h_{(l)}$ for $r = l(n- 1) + 1$ is given by $h_{(l)}(u_1^r) = \underbrace{h(h(\ldots, h(h}_l(u^n _1), u_{n+1}^{2n -1}),\ldots), u_{r-n+1}^{r})$. An $n$-ary hypergroupoid $(H, h)$ is commutative if $h(u_1^n) = h(u_{\sigma(1)}^{\sigma(n)})$ for all $u_1^n \in H $ and $ \sigma \in \mathbb{S}_n$.
An $n$-ary semihypergroup is an $n$-ary hypergroupoid $(H, h)$, which is associative, that is $h(u^{i-1}_1, h(u_i^{n+i-1}), u^{2n-1}_{n+i}) = h(u^{j-1}_1, h(u_j^{n+j-1}), u_{n+j}^{2n-1})$  for  $1 \leq i < j \leq n$ and  $u_1^{2n-1} \in H$. If the equation $x \in h(u_1^{i-1}, x_i, u_{ i+1}^n)$ in an $n$-ary hypergroupoid $(H, h)$ has a solution $x_i \in H$
for all  $u_1^{i-1}, u_{ i+1}^n,x \in H$ and $i \in \{1,\ldots,n\}$, then $(H,h)$ is called an $n$-ary quasihypergroup. An  $n$-ary hypergroup is an $n$-ary semihypergroup that is an $n$-ary quasihypergroup. A non-empty subset $G$ of an $n$-ary hypergroup $H$ is called
an $n$-ary subhypergroup of $H$ if $(G,h)$ is an $n$-ary hypergroup \cite{d1}.

\begin{definition}
\cite{d1} A triple $(H, h, k)$, or simply $H$, refers to  a Krasner $(m, n)$-hyperring  if it satisfies the following axioms:
\begin{itemize} 
\item[\rm{(a)}]~ $(H, h$) is a canonical $m$-ary hypergroup, that is 
\begin{itemize} 
\item[\rm{(1)}]~there exists a unique $e \in H$ with  $h(u, e^{(m-1)}) = \{u\}$  for each $u \in H$ ;
\item[\rm{(2)}]~for all $u \in H$ there exists a unique $u^{-1} \in H$ with  $e \in h(u, u^{-1}, e^{(m-2)})$;
\item[\rm{(3)}]~if $u \in h(u^m _1)$, then  $u_i \in h(u, u^{-1},\ldots, u^{-1}_{ i-1}, u^{-1}_ {i+1},\ldots, u^{-1}_ m)$ for each $i\in \{1,\ldots,m\}$.
\end{itemize} 
\item[\rm{(b)}]~ $(H, k)$ is a $n$-ary semigroup;
\item[\rm{(c)}]~
$k(u^{i-1}_1, h(v^m _1 ), u^n _{i+1}) = h(k(u^{i-1}_1, v_1, a^n_{ i+1}),\ldots, k(u^{i-1}_1, v_m, u^n_{ i+1}))$ for every $u^{i-1}_1 , u^n_{ i+1}, v^m_ 1 \in H$, and $i \in \{1,\ldots,n\}$;
\item[\rm{(d)}]~  $k(0, u^n _2) = k(u_2, 0, u^n _3) = \cdots =k(u^n_ 2, 0) = 0$ for all $u^n_ 2 \in H$.
\end{itemize} 
\end{definition}
We mention here that, throughout this paper   $H$ is a commutative Krasner $(m,n)$-hyperring with scalar identity $1_H$, that is $k(u,1_H^{(n-1)})=u$ for all $u \in H$. 

If $(G, h, k)$ is a Krasner $(m, n)$-hyperring such that $\varnothing \neq  G \subseteq H$, then  $G$ is called a subhyperring of $H$. If $\varnothing \neq I \subseteq H$ such that $(I, h)$ is an $m$-ary subhypergroup
of $(H, h)$ and $k(u^{i-1}_1, I, u_{i+1}^n) \subseteq I$  for  $u^n _1 \in H$, $i \in \{1,\ldots,n\}$, then $I$ is said to be  a hyperideal of $H$.
For any hyperideal $I$   of $H$,   the set $H/I=\{h(u_1^{i-1},I,u_{i+1}^n) \ \vert \ u_1^{i-1},u_{i+1}^n \in H\}$ is a Krasner $(m, n)$-hyperring with $m$-ary hyperoperation and $n$-hyperoperation $h$ and $k$, respectively \cite{sorc1}.
\begin{definition}
\cite{sorc1} A proper hyperideal $Q$ of $H$ refers to an  $n$-ary prime hyperideal if $k(Q_1^ n) \subseteq Q$ for hyperideals $Q_1^n$ in $H$ implies that $Q_i \subseteq Q$ for some $i \in \{1,\cdots,n\}$. 
\end{definition}
Lemma 4.5 in \cite{sorc1} verifies that $Q$ is an $n$-ary prime hyperideal of $H$ if  $k(u^n_ 1) \in Q$ for all $u^n_ 1 \in H$ implies $u_i \in Q$ for some $i \in \{1,\cdots,n\}$.

  The radical of a hyperideal $I$ of $H$ is 
the intersection of all $n$-ary prime hyperideals of $H$ containing $I$ and it is   denoted by $rad(I)$. We consider $rad(I)=H$ if there is no  $n$-ary prime hyperideal  containing $I$. By Theorem 4.23 in \cite{sorc1} we obtain
\[rad(I)= \biggm{\{} u \in H \ \vert \ \biggm{\{} 
\begin{array}{lr}
k(u^{(r)},1_H^{(n-r)}) \in I,& r \leq n\\
k_{(l)}(u^{(r)}) \in I, & r>n, \ r=l(n-1)+1
\end{array}
\biggm{\}} \biggm{\}}.\]
\begin{definition}
A proper hyperideal $P$ of $H$ refers to   an  $n$-ary primary hyperideal if $k(u^n _1) \in P$ implies that $u_i \in P$ or $k(u_1^{i-1}, 1_H, u_{ i+1}^n) \in rad(P)$ for some $i \in \{1,\cdots,n\}$. 
\end{definition}

\begin{definition} \cite{sorc1} For any  $u \in H$,  $\langle u \rangle$ denotes the hyperideal generated by $u$  and defined by $\langle u \rangle=k(u,H,1_H^{(n-2)})=\{k(u,v,1_H^{(n-2)}) \ \vert \ v \in H\}$. Moreover, for any hyperideal $I$ of $H$ we define $(I:u)=\{v \in H \ \vert \ k(u,v,1_H^{(n-2)}) \in I\}$.
\end{definition}
\begin{definition} \cite{sorc1} A hyperideal $M$ of $H$ refers to a  maximal hyperideal if  $M \subseteq I \subseteq H$ for every hyperideal $I$ of $H$ implies that $I=M$ or $I=H$.
\end{definition}
$Max(H)$ denotes the set of all maximal hyperideal of $H$.  Moreover, if $H$ has just one maximal hyperideal, then $G$ is called local.
\begin{definition} \cite{sorc1} 
 An element $u \in H$ is  invertible if there exists $v \in H$ with $1_H=k(u,v,1_H^{(n-2)})$. 
\end{definition}






\begin{definition} \cite{d1}
Let $(H_1, h_1, k_1)$ and $(H_2, h_2, k_2)$ be two Krasner $(m, n)$-hyperrings. A mapping
$\eta : H_1 \longrightarrow H_2$ is called a homomorphism if for all $u^m _1,  v^n_ 1 \in H_1$ we have
\begin{itemize} 
\item[\rm{(i)}]~$\eta(h_1(u_1,\ldots, u_m)) = h_2(\eta(u_1),\ldots,\eta(u_m)),$
\item[\rm{(ii)}]~$\eta(k_1(v_1,\ldots, v_n)) = k_2(\eta(v_1),\ldots,\eta(v_n)), $
\item[\rm{(iii)}]~$\eta(1_{H_1})=1_{H_2}.$
\end{itemize}
\end{definition}

\section{  $n$-ary Endo-prime hyperideals }
In this section, we treat to the introducing $Endo$-prime hyperideal associated with $\theta$ on a Krasner $(m,n)$-hyperring $H$  where $\theta : H \longrightarrow H$   is an   endomorphism and investigate many results with respect to such hyperideals.
\begin{definition} 
Let $E$ be a proper hyperideal of $H$ and $\theta : H \longrightarrow H$ be an   endomorphism. We say that $E$ is an $n$-ary $Endo$-prime hyperideal associated with $\theta$, if for all $u_1^n \in H$, $k(u_1^n) \in E$ implies that $u_i \in E$ or $\theta\big(k(u_ 1^{i-1},1_H,u_{i+1}^n)\big) \in E$ for some $i \in \{1,\ldots,n\}$. 
\end{definition}
\begin{example}
 Let G=\{0,1,u,v\} and $\mathcal{P}^*(G)$ be the family of all non-empty subsets of $H$. Consider the hyperintegral domain $(G,\oplus,\circ)$ where $\oplus: G \times G \longrightarrow  \mathcal{P}^*(G)$ is the multi-valued function defined by

\[ \begin{tabular}{|c|c|c|c|c| } 
\hline   $\oplus$  & $0$ & $1$ & $u$ & $v$  
\\ \hline $0$ & $\{0\}$ & $\{1\}$ & $\{u\}$ & $\{v\}$  
\\ \hline $1$ & $\{1\}$ & $\{0,u,v\}$ & $\{1,u\}$ & $\{1,v\}$   
\\ \hline$u$ & $\{u\}$ & $\{1,u\}$ & $\{0,1,v\}$ & $\{u,v\}$  
\\ \hline $v$ & $\{v\}$ & $\{1,v\}$ & $\{u,v\}$ & $\{0,1,u\}$  
\\ \hline
\end{tabular}\]

and the multiplication $\circ$ defined by 

\[ \begin{tabular}{|c|c|c|c|c|c|c|} 
\hline $\circ$ & $0$ & $1$ & $u$ & $v$  
\\ \hline $0$ & $0$ & $0$ & $0$ & $0$  
\\ \hline $1$ & $0$ & $1$ & $u$ & $v$  
\\ \hline$u$ & $0$ & $u$ & $v$ & $1$  
\\ \hline $v$ & $0$ & $v$ & $1$ & $u$  
\\ \hline
\end{tabular}\]

Put $H=R \times G$ where $R$ is a Krasner $(2,2)$-hyperring. Then $E=R \times 0$ is an $2$-ary $Endo$-prime hyperideal associated with $\theta_R \times \theta_G$ such that $\theta_R: R \longrightarrow R$ is an endomorphism and $\theta_G: G \longrightarrow G$ is the inclusion homomorphism.
\end{example}
\begin{remark} \label{haji}
Let $H$ be a Krasner $(m,n)$-hyperring and $\theta : H \longrightarrow H$ be an endomorphism. Every $n$-ary prime hyperideal of $H$ is $n$-ary $Endo$-prime hyperideal associated with $\theta$.
\end{remark}
The converse of Remark \ref{haji}  may not be always true as it is shown in the following example.
\begin{example} \label{salami}
Consider the set  $H=\{0,1,2,3\}$ with the $2$-ary hyperopertion $\oplus$ defined as 

\[\begin{tabular}{|c|c|c|c|c|} 
\hline   $\oplus$  & $0$ & $1$ & $2$ & $3$ 
\\ \hline $0$ & $\{0\}$ & $\{1\}$ & $\{2\}$ & $\{3\}$  
\\ \hline $1$ & $\{1\}$ & $\{0,1\}$ & $\{3\}$ & $\{2,3\}$ 
\\ \hline$2$ & $\{2\}$ & $\{3\}$ & $\{0\}$ & $\{1\}$  
\\ \hline $3$ & $\{3\}$ & $\{2,3\}$ & $\{1\}$ & $\{0,1\}$  
\\ \hline
\end{tabular}\]

and the 4-ary operation $k$ defined as $k(u_1^4)=2$ if $u_1^4 \in \{2,3\}$ or $0$ if otherwise. By Example 4.8 in \cite{sorc1}, $H$ is a Krasner $(2,4)$-hyperideal. Assume that $\theta : H \longrightarrow H$ is an endomorphism. 
Then, $0$ is an $4$-ary $Endo$-prime hyperideal associated with $\theta$. However,  $0$ is not an $4$-ary prime hyperideal of $H$ since $k(1,2,3,3)=0$ but $1,2,3 \notin 0$.

\end{example}
 Our first theorem of this section establishes that the radical of an $n$-ary $Endo$-prime hyperideal associated with $\theta$ is an $n$-ary $Endo$-prime hyperideal associated with $\theta$.
\begin{theorem} \label{1}
Let $E$ be a proper hyperideal of $H$ and $\theta : H \longrightarrow H$ be an   endomorphism. If $E$ is   an $n$-ary $Endo$-prime hyperideal associated with $\theta$, then $rad(E)$ is an $n$-ary $Endo$-prime hyperideal associated with $\theta$.
\end{theorem} 
\begin{proof} 
Assume that $k(u_1^n) \in rad(E)$ for $u_1^n \in H$. Then there exists $r \in \mathbb{N}$ such that if $r \leq n$, then $k\big(k(u_1^n)^{(r)},1_H^{n-r}\big) \in E$. Therefore, by associativity we get\\

$\hspace{1.2cm}k\big(u_i^{(r)},k(u_1^{i-1},1_H,u_{i+1}^n)^{(r)},1_H^{(n-2r)} \big) \in E$

$\hspace{0.5cm} \Longrightarrow k\big(u_i^{(r)},k(u_1^{i-1},1_H,u_{i+1}^n)^{(r)},k(1_H^{(n)}),1_H^{(n-2r-1)} \big) \in E$

$\hspace{0.5cm} \Longrightarrow k\big(u_i^{(r)},k(k(u_1^{i-1},1_H,u_{i+1}^n)^{(r)},1_H^{(n-r)}),1_H^{(n-r-1)} \big) \in E$\\

This implies that $k\big(k(u_1^{i-1},1_H,u_{i+1}^n)^{(r)},1_H^{(n-r)}\big) \in E$ or $k\big(\theta(u_i)^{(r)},1_H^{(n-r)}\big)=\theta \big( k(u_i^{(r)},1_H^{(n-r)})\big) \in E$ as $E$ is   an $n$-ary $Endo$-prime hyperideal associated with $\theta$. Then we conclude that $k(u_1^{i-1},1_H,u_{i+1}^n) \in rad(E)$ or $\theta(u_i) \in rad(E)$. This means that $rad(E)$ is an $n$-ary $Endo$-prime hyperideal associated with $\theta$. If $r=l(n-1)+1$, then we have $k_{(l)}(k(u_1^n)^{(r)}) \in E$ and  similar to the before part we get the result that $rad(E)$ is an $n$-ary $Endo$-prime hyperideal associated with $\theta$.
\end{proof} 
The following result shows that every $n$-ary $Endo$-prime hyperideal of $H$ associated to $\theta$ contains $\theta(E)$.
\begin{proposition} \label{2}
Let $\theta : H \longrightarrow H$  be an   endomorphism. If  $E$ is an $n$-ary $Endo$-prime hyperideal of $H$ associated to $\theta$, then $\theta(E) \subseteq E$.
\end{proposition}
\begin{proof}
Let $u \in E$. Then we have $ k(u,1_H^{(n-1)}) \in E$ which implies $1_H \in E$ or $\theta(u)=\theta \big(k(u,1_H^{(n-1)}) \big) \in E$. Since the first case is impossible, we have $\theta(u) \in E$ and so $\theta(E) \subseteq E$.
\end{proof}
The next example shows that the converse of Proposition \ref{2} is not
true, in general.
\begin{example} \label{pelarak}
Consider the Krasner $(2,2)$-hyperfield $(H=\{0,1,u,v,w\}, \boxplus, \circ)$ where the hyperaddition $ \boxplus$ and multiplication $\circ$ defiened by

\[\begin{tabular}{|c|c|c|c|c|c|c|} 
\hline   $\boxplus$  & $0$ & $1$ & $u$ & $v$ & $w$
\\ \hline $0$ & $\{0\}$ & $\{1\}$ & $\{u\}$ & $\{v\}$ & $\{w\}$
\\ \hline $1$ & $\{1\}$ & $\{1\}$ & $\{1,u\}$ & $H$ & $\{1,w\}$ 
\\ \hline $u$ & $\{u\}$ & $\{1,u\}$ & $\{u\}$ & $\{u,v\}$ & $H$ 
\\ \hline $v$ & $\{v\}$ & $H$ & $\{u,v\}$ & $\{v\}$ & $\{v,w\}$
\\ \hline $w$ & $\{w\}$ & $\{1,w\}$ & $H$ & $\{v,w\}$ & $\{w\}$
\\ \hline
\end{tabular}\]

\[\begin{tabular}{|c|c|c|c|c|c|c|} 
\hline $\circ$ & $0$ & $1$ & $u$ & $v$ & $w$
\\ \hline $0$ & $0$ & $0$ & $0$ & $0$ & $0$  
\\ \hline $1$ & $0$ & $1$ & $u$ & $v$ & $w$
\\ \hline $u$ & $0$ & $u$ & $v$ & $w$ & $1$
\\ \hline $v$ & $0$ & $v$ & $w$ & $1$ & $u$
\\ \hline $w$  & $0$ & $w$ & $1$ &  $u$ & $v$
\\ \hline
\end{tabular}\]

Define the endomorphism $\theta : H \times H \longrightarrow H \times H$ by $(x,y) \mapsto (y,x)$. The hyperideal $E=(0,0)$ of $H \times H$ is not an $2$-ary $Endo$-prime hyperideal associated with $\theta$ as $(1,0) \circ (0,1) \in E$ but neither $(1,0), (0,1) \in E$ nor $\theta(1,0), \theta(0,1) \in E$. However, $\theta(E) \subseteq E$.
\end{example}
\begin{proposition} \label{bagheri}
Let $E$ be a proper hyperideal of $H$, $\theta : H \longrightarrow H$  an   endomorphism and $u \in H$. If $E$ is   an $n$-ary $Endo$-prime hyperideal associated with $\theta$, then 
\begin{itemize} 
\item[\rm{(1)}]~ $E^{\prime}=\{u \in H \ \vert \ \theta(u) \in E\}$ is an $n$-ary $Endo$-prime hyperideal associated with $\theta$.
\item[\rm{(2)}]~ $(E:A)$ is an $n$-ary $Endo$-prime hyperideal associated with $\theta$ where $A$ is a subset of $H$.
\item[\rm{(3)}]~$k \big(u^ {(r)} , 1_H^{(n-r)} ) \big) \in E$ with $r \leq n$, or $  k_{(l)} (u^ {(r)} )  \in E$ with $r = l(n-1) + 1$ implies that $\theta (u) \in E$.

\item[\rm{(4)}]~ $k \big( \theta(u)^ {(r)} , 1_H^{(n-r)} ) \big) \in E$ with $r \leq n$, or $  k_{(l)} (\theta(u)^ {(r)} )  \in E$ with $r = l(n-1) + 1$ implies that $\theta^2(u) \in E$.
\end{itemize}
\end{proposition}
\begin{proof}

(1) Let $k(u_1^n) \in E^{\prime}$ for $u_1^n \in H$. Then we get $k \big (\theta(u_1),\cdots,\theta(u_n) \big)=\theta \big(k(u_1^n) \big) \in E$. Since $E$ is   an $n$-ary $Endo$-prime hyperideal associated with $\theta$, we have $\theta(u_i) \in E$ for some $i \in \{1,\ldots,n\}$ or we get $\theta \big( \theta \big( k(u_1^{i-1},1_H,u_{i+1}^n) \big) \big)=\theta \big(k \big( \theta(u_1),\cdots,\theta(u_{i-1}),1_H,\theta(u_{i+1}),\cdots,\theta(u_n) \big) \big) \in E$ . Hence, we conclude that $u_i \in E^{\prime}$ or $\theta \big( k(u_1^{i-1},1_H,u_{i+1}^n) \big) \in E^{\prime}$. Consequently, $E^{\prime}$ is an $n$-ary $Endo$-prime hyperideal associated with $\theta$.

(2) Let $A$ be a subset of $H$ and $k(u_1^n) \in (E :A)$ for $u_1^n \in H$. Therefore, we have $k(1_H,u_2^n) \in \big (E : k(u_1,A,1_H^{(n-2)}) \big)=(E:u_1) \cup (E:A)$. Hence $k(u_1^n) \in E$ or $k( 1_H,u_2^n) \in (E : A)$. In the first case, we get  $u_i \in E$ or $\theta \big( k(u_1^{i-1},1_H,u_{i+1}^n) \big) \in E$ for some $i \in \{1,\ldots,n\}$ as $E$ is   an $n$-ary $Endo$-prime hyperideal associated with $\theta$. In the second case, we have  $k( 1_H,u_2^n) \in (E : A)$ and so $\theta \big(k( 1_H,u_2^n) \big) \in (E : A)$ by (1). Thus, we obtain that $u_i \in (E:A)$ or $\theta \big( k(u_1^{i-1},1_H,u_{i+1}^n) \big) \in (E : A)$, that is, $(E:A)$ is an $n$-ary $Endo$-prime hyperideal associated with $\theta$.

(3) It is obvious.

(4) Let   $u \in H$ and $k \big( \theta(u)^ {(r)} , 1_H^{(n-r)} ) \big) \in E$ with $r \leq n$, or $ \big(k_{(l)} (\theta(u)^ {(r)} ) \big) \in E$ with $r = l(n-1) + 1$. Put $v=\theta(u)$. So, we have $k \big(v^ {(r)} , 1_H^{(n-r)} ) \big) \in E$ with $r \leq n$, or $ \big(k_{(l)} (v^ {(r)} ) \big) \in E$ with $r = l(n-1) + 1$. By (4), we conclude that $\theta(v) \in E$ which means $\theta^2(u) \in E$.
\end{proof}
Our next theorem characterizes the behavior of hyperideals in relation to $E$ and how to ensure that $E$ is an $Endo$-prime hyperideal in the Krasner $(m,n)$-hyperring $H$.
\begin{theorem} \label{madar}
Let $E$ be a proper hyperideal of $H$ and $\theta : H \longrightarrow H$  be an   endomorphism. Then  the following are equivalent:
\begin{itemize} 
 \item[\rm{(1)}]~ $E$ is an $n$-ary $Endo$-prime hyperideal associated to $\theta$.
\item[\rm{(2)}]~ $\langle u \rangle \subseteq E$ or $\theta(E:u) \subseteq E$ for all $u \in H$.
\item[\rm{(3)}]~ $k(E_1^n) \subseteq E$ for all hyperideals $E_1^n$ of $H$ implies $E_i \subseteq E$ or $\theta \big(k(E_1^{i-1},1_H,E_{i+1}^n) \big) \subseteq E$ for some $i \in \{1,\ldots,n\}$.
\item[\rm{(4)}]~ $E=\big(E:k(u_1^{i-1},1_H,u_{i+1}^n) \big)$ for all $u_1^{i-1},u_{i+1}^n \in H$ with $\theta \big(k(u_1^{i-1},1_H,u_{i+1}^n) \big) \notin E.$ 
\end{itemize}
\end{theorem}
\begin{proof}
(1) $\Longrightarrow$ (2) It is clear by Proposition \ref{2} and Proposition \ref{bagheri} (2).

(2) $\Longrightarrow$ (3) $k(E_1^n) \subseteq E$ for all hyperideals $E_1^n$ of $H$ such that $E_i \nsubseteq E$ for any $i \in \{1,\ldots,n\}$. Take any $u_i \in E_i$ for some $i \in \{1,\ldots,n\}$. Then we get $k(u_1^{i-1},1_H,u_{i+1}^n) \in (E:u_i)$ for all $u_j \in E_j$ where $j \in \{1,\ldots,i-1,i+1,\ldots,n\}$. By the hypothesis, we obtain that $\theta \big(k(u_1^{i-1},1_H,u_{i+1}^n) \big) \subseteq  \theta (E:u_i) \subseteq E$ which means $\theta \big(k(E_1^{i-1},1_H,E_{i+1}^n) \big) \subseteq E$.

(3) $\Longrightarrow$ (4) Let  $\theta \big(k(u_1^{i-1},1_H,u_{i+1}^n) \big) \notin E $ for $u_1^{i-1},u_{i+1}^n \in H$ and $I=\big( E:k(u_1^{i-1},1_H,u_{i+1}^n)\big)$. Since $ k \big(\langle k(u_1^{i-1},1_H,u_{i+1}^n) \rangle, I, 1^{(n-2)} \big) \subseteq E$, we conclude that $k \big (\langle u_1 \rangle, \ldots,\langle u_{i-1} \rangle,  I, \langle u_{i+1} \rangle, \ldots,\langle u_n \rangle \big) \subseteq E$. Since $\theta \big(k(u_1^{i-1},1_H,u_{i+1}^n) \big) \notin E $, we obtain  $I=\big( E:k(u_1^{i-1},1_H,u_{i+1}^n)\big) \subseteq E$ as needed.

(4) $\Longrightarrow$ (1) Let $k(u_1^n) \in E$ for $u_1^n \in H$ but $\theta \big(k(u_1^{i-1},1_H,u_{i+1}^n) \big) \notin E$ for $i \in \{1,\ldots,n\}$. It follows that $u_i \in \big( E: k(u_1^{i-1},1_H,u_{i+1}^n) \big)$. Hence, we get $u_i \in E$ by the assumption. Consequently, $E$ is an $n$-ary $Endo$-prime hyperideal associated to $\theta$.
\end{proof}
Recall that, in general, the intersection of a family of $Endo$-prime hyperideals  associated to $\theta$ is not $Endo$-prime hyperideal, but we have the following results.
\begin{theorem} \label{inter}
Let $\eta:H_1 \longrightarrow H_2$ be a homomorphism and $\{E_j\}_{j \in \omega}$ is a chain of hyperideals of $H$. If $E_j$ is an  $n$-ary $Endo$-prime hyperideal   associated to $\theta$ for every $j \in \omega$, then so is $\cap_{j \in \omega} E_j$.
\end{theorem}
\begin{proof}
Let $k(u_1^n) \in \cap_{j \in \omega} E_j$ for $u_1^n \in H$ such that neither $u_i \in \cap_{j \in \omega} E_j$ nor $\theta(u_1^{i-1},1_H,u_{i+1}^n) \in \cap_{j \in \omega} E_j$ for all $i \in \{1,\ldots,n\}$. Then, we conclude that   $u_i \notin E_a$ and $\theta(u_1^{i-1},1_H,u_{i+1}^n) \notin E_b$ for some $a,b \in \omega$. Let $E_a \subseteq E_b$. In this case, we have  $k(u_1^n) \in E_a$ but $u_i \notin E_a$ and $\theta(u_1^{i-1},1_H,u_{i+1}^n) \notin E_a$. This is impossible as $E_a$ is an  $n$-ary $Endo$-prime hyperideal   associated to $\theta$. Now, let $E_b \subseteq E_a$. In this case, we face the fact that $k(u_1^n) \in E_b$ and neither $u_i \in E_b$ nor $\theta(u_1^{i -1},1_H,u_{i +1}^n) \in E_b$. This is a contradiction since $E_b$ is an  $n$-ary $Endo$-prime hyperideal   associated to $\theta$. Thus, $\cap_{j \in \omega} E_j$ is an  $n$-ary $Endo$-prime hyperideal   associated to $\theta$.
\end{proof}
\begin{theorem}\label{2.01}
Let  $\theta : H \longrightarrow H$  be an   endomorphism and $E$ be an $n$-ary $Endo$-prime hyperideal of $H$ associated with $\theta$. Then $E$ contains $\theta(Q)$ for every minimal $n$-ary prime hyperideal $Q$ over $E$.
\end{theorem}
\begin{proof}
Suppose that $E$ is an $n$-ary $Endo$-prime hyperideal of $H$ associated with $\theta$ and $Q$ is an $n$-ary prime hyperideal of $H$  that is minimal over $E$. We show that $E$ contains $\theta(Q)$.  Let us assume that $u \in Q$ but $\theta(u) \notin E$. Then there exists $v \in H \backslash Q$ and a nonnegative integer $r$ such that $k \big( k(u^{(r)},1_H^{(n-r)}),v,1_H^{(n-2)}) \in E$ for $r  \leq n$ or $k \big (k_{(l)}(u^{(r)}),v,1_H^{(n-2)} \big) \in E$ for $r=l(n-1)+1$ by Lemma 3.5 in \cite{mah2}. In the first case, since $E$ is an $n$-ary $Endo$-prime hyperideal of $H$ associated with $\theta$, $k \big(u,k(u^{(r-1)},v,1_H^{(n-r+1)}),1_H^{(n-2)} \big) \in E$ and $\theta(u) \notin E$, we conclude that $k(u^{(r-1)},v,1_H^{(n-r+1)}) \in E$ and by continuing the process, we get $k(u,v,1^{(n-1)}) \in E$. Since $\theta(u) \notin E$, we have $ v \in E \subseteq Q$ which is impossible. Hence, we get the result that $E$ contains $\theta(Q)$.  In the second case, the claim follows by using a similar argument to the previous part.
\end{proof}
As an immediate consequence of the previous theorem, we have the following result.

\begin{corollary}
Let  $\theta : H \longrightarrow H$  be an   endomorphism. If $E$ is   an $n$-ary $Endo$-prime hyperideal associated with $\theta$ and $Q$ is an $n$-ary prime hyperideal   of $H$ that is minimal over $E$, then $\theta(E) \subseteq \theta(Q) \subseteq E \subseteq Q$.
\end{corollary}
\begin{theorem}
Let  $\theta : H \longrightarrow H$  be an   endomorphism, $E$   a hyperideal of $H$ and $\theta(Q)=Q$ for each $n$-ary prime hyperideal $Q$ of $H$. Then  $E$ is an $n$-ary $Endo$-prime hyperideal associated with $\theta$ if and only if $E$ is an $n$-ary prime hyperideal.
\end{theorem}
\begin{proof}
$\Longrightarrow$ Let $E$ be an $n$-ary $Endo$-prime hyperideal associated with $\theta$. Therefore, we have $Q=\theta(Q) \subseteq E$ for each $n$-ary prime hyperideal of $H$ that is minimal over $E$ by Theorem \ref{2.01}. Hence, we get $E=Q$ which means $E$ is an $n$-ary prime hyperideal.

$\Longleftarrow$ It is obvious.
\end{proof}

\begin{proposition} \label{2.1}
Let $\theta : H \longrightarrow H$  be an   endomorphism. If $E$ is   an $n$-ary $Endo$-prime hyperideal associated with $\theta$, then  $\theta(rad(E)) \subseteq E$.
\end{proposition}
\begin{proof}
Let $u \in rad(E)$ such that $u \in E$. Then we get $\theta(u) \in E$ by Proposition \ref{2}. Now, suppose that  $u \notin E$. Since  $u \in rad(E)$, there exists a minimal element $r \in \mathbb{N}$ with $k(u^{(r)},1_H^{(n-r)}) \in E$ for $r \leq n$, or $k_{(l)}(u^{(r)}) \in E$ for $r=l(n-1)+1$. In the first case, we have $k \big(k(u^{(r-1)},1_H^{(n-r+1)}),u,1_H^{(n-2)}\big) \in E$ which means $\theta(u) \in E$ as $E$ is   an $n$-ary $Endo$-prime hyperideal associated with $\theta$ and $r$ is a minimal integer.  Thus, $\theta(rad(E)) \subseteq E$. In the second case, similarly, we can show that $E$ contained $\theta(rad(E))$.
\end{proof}
Recall from \cite{mah6} that an element $u \in H$ is nilpotent, if 
there exists $r \in \mathbb {N}$ with $k(u^ {(r)} , 1_H^{(n-r)} )=0$ for $r \leq n$, or $k_{(l)} (u^ {(r)} )$ for $r = l(n-1) + 1$.  The set of all nilpotent elements of $H$ is denoted by $\Upsilon_{\theta}(H)$. Since  $rad(0)$ is equal to  the set of all nilpotent elements of $H$ and it is in the radical of every hyperideal of $H$, we have the following result by Proposition \ref{2.1} .
\begin{corollary} \label{2.2}
If $\theta : H \longrightarrow H$  is an   endomorphism and $E$ is   an $n$-ary $Endo$-prime hyperideal associated with $\theta$, then $\theta(\Upsilon(H)) \subseteq E$. 
\end{corollary}

\begin{theorem}\label{3}
Assume that $\eta:H_1 \longrightarrow H_2$ is a homomorphism and $\theta \in End(H_1) \cap End(H_2)$ commutes with $\eta$. If the hyperideal $E_2$ of $H_2$ is an $n$-ary $Endo$-prime hyperideal associated with $\theta$, then the hyperideal $\eta^{-1}(E_2)$ of $H_1$ is an $n$-ary $Endo$-prime hyperideal associated with $\theta$.
\end{theorem}
\begin{proof}
Suppose that the hyperideal $E_2$ of $H_2$ is an $n$-ary $Endo$-prime hyperideal associated with $\theta$ and $k_1(u_1^n) \in \eta^{-1}(E_2)$ for $u_1^n \in H_1$. Therefore we obtain $k_2 \big( \eta(u_1),\cdots,\eta(u_n) \big)=\eta \big(k_1(u_1^n) \big) \in E_2$. Since the hyperideal $E_2$ of $H_2$ is an $n$-ary $Endo$-prime hyperideal associated with $\theta$, we conclude that $\eta(u_i) \in E_2$ for some $i \in \{1,\ldots,n\}$ or $\eta \big( \theta \big( k_1(u_1^{i-1},1_{H_1},u_{i+1}^n) \big) \big)=\theta \big( \eta \big( k_1(u_1^{i-1},1_{H_1},u_{i+1}^n) \big) \big)=\theta \big(k_2 \big( \eta(u_1),\cdots,\eta(u_{i-1}),1_{H_2},\eta(u_{i+1}),\cdots,\eta(u_n) \big)\big) \in E_2$. This implis that $u_i \in \eta^{-1}(E_2)$ or $\theta \big( k_1(u_1^{i-1},1_{H_1},u_{i+1}^n) \big) \in \eta^{-1}(E_2)$. Thus, the hyperideal $\eta^{-1}(E_2)$ of $H_1$ is an $n$-ary $Endo$-prime hyperideal associated with $\theta$.
\end{proof}
\begin{proposition}\label{4}
Let $\theta : H \longrightarrow H$ be an   endomorphism. Then $Ker \theta$ is in the intersection of all $n$-ary $Endo$-prime hyperideals associated with $\theta$.
\end{proposition}
\begin{proof}
Let $\theta : H \longrightarrow H$ be an   endomorphism and $a \in Ker \theta$. So, $\theta(a)=0$ which means $\theta(a)$ is in every $n$-ary $Endo$-prime hyperideal  associated with $\theta$. Since the inverse image of each    $n$-ary $Endo$-prime hyperideal associated with $\theta$ is also an $n$-ary $Endo$-prime hyperideal associated with $\theta$ by Theorem \ref{3}, we conclude that $a$ is in the intersection of all $n$-ary $Endo$-prime hyperideals associated with $\theta$. 
\end{proof}
Recall from \cite{sorc1} that a commutative Krasner $(m,n)$-hyperring $H$ is an $n$-ary hyperintegral domain if $k(u_1^n)=0$ for $u_1^n \in H$ implies that $u_i=0$ for some $i \in \{1,\ldots,n\}$.
We say that the element $u \in H$ is $\theta$-nilpotent, if 
there exists $r \in \mathbb {N}$ with $\theta \big( k(u^ {(r)} , 1_H^{(n-r)} ) \big)=0$ for $r \leq n$, or $\theta \big(k_{(l)} (u^ {(r)} ) \big)=0$ for $r = l(n-1) + 1$. The $\theta$-nilradical of $H$, denoted by $\Upsilon_{\theta}(H)$, is the set of all $\theta$-nilpotent elements of $H$.
\begin{theorem} \label{5}
Assume that $H$ is an $n$-ary hyperintegral domain.
Then $\Upsilon_{\theta}(H)$ is equal to  the intersection of all    $n$-ary $Endo$-prime hyperideals of $H$  associated with  $\theta$.
\end{theorem}
\begin{proof}
Let $u \in \Upsilon_{\theta}(H)$. Then there exists $r \in \mathbb {N}$ with $\theta \big( k(u^ {(r)} , 1_H^{(n-r)} ) \big)=0$ for $r \leq n$, or $\theta \big(k_{(l)} (u^ {(r)} ) \big)=0$ for $r = l(n-1) + 1$. Hence $k(u^ {(r)} , 1_H^{(n-r)} ) \in Ker \theta$ or $ k_{(l)} (u^ {(r)} ) \in Ker \theta$. Since $Ker \theta$ is an $n$-ary prime hyperideal of $H$ by Proposition \ref{4}, we obtain $u \in Ker \theta$. This implies that $u$ is in the intersection of all    $n$-ary $Endo$-prime hyperideals of $H$  associated with  $\theta$ by Proposition \ref{3} and so  $\Upsilon_{\theta}(H)$ is in the intersection of all    $n$-ary $Endo$-prime hyperideals of $H$  associated with  $\theta$. Now, suppose that $u$ belongs to the intersection of all    $n$-ary $Endo$-prime hyperideals of $H$  associated with  $\theta$ but $u \notin \Upsilon_{\theta}(H)$. Let   $\mathcal{HI}(H)$ be the set of all hyperideals of $H$. Put
\[\Gamma= \biggm{\{} P \in \mathcal{HI}(H) \ \vert \   \biggm{\{} 
\begin{array}{lr}
k(u^{(r)},1_H^{(n-r)}) \notin P,& r \leq n\\
k_{(l)}(u^{(r)}) \notin P, & r>n, \ r=l(n-1)+1
\end{array}
\biggm{\}}  \biggm{\}}.\]
Then $\Gamma \neq \varnothing$ because $0 \in \Gamma$. So,
$\Gamma$ is a partially ordered set with
respect to set inclusion relation. Then, there is a hyperideal
$E$ which is maximal in $\Gamma$, by Zorn$^,$s lemma. We show that $E$ is an $n$-ary $Endo$-prime hyperideal of $H$ associated with $\theta$. Assume that $\theta(u_1),\ldots,\theta(u_n) \notin E$ for $u_1^n \in H$. Thus, we have $E \subseteq h(E,\langle \theta(u_i) \rangle, 0^{(m-2)})$ for every $i \in \{1,\ldots,n\}$. Since $h(E,\langle \theta(u_i) \rangle, 0^{(m-2)}) \notin \Gamma$ for every $i \in \{1,\ldots,n\}$, there exists $r_i \in \mathbb {N}$ with $\theta \big( k(u^ {(r_i)} , 1_H^{(n-r_i)} ) \big) \in h(E,\langle \theta(u_i) \rangle, 0^{(m-2)})$ for $r_i \leq n$, or $\theta \big(k_{(l_i)} (u^ {(r_i)} ) \big) \in h(E,\langle \theta(u_i) \rangle, 0^{(m-2)})$ for $r_i = l_i(n-1) + 1$. Put $r=r_1+\cdots+r_n$. Hence, we conclude that $\theta \big( k(u^ {(r)} , 1_H^{(n-r)} ) \big) \in h\big(E,\langle \theta(k(u_1^n)) \rangle, 0^{(m-2)}\big)$ for $r \leq n$, or $\theta \big(k_{(l)} (u^ {(r)} ) \big) \in h\big(E,\langle \theta(k(u_1^n)) \rangle, 0^{(m-2)}\big)$ for $r = l(n-1) + 1$. This means that $h\big(E,\langle \theta(k(u_1^n)) \rangle, 0^{(m-2)}\big) \notin \Gamma$ which implies $\theta(k(u_1^n))=k\big(\theta(u_1),\cdots,\theta(u_n)\big) \notin E$. Then $k(u_1^n) \notin E$ and $E$ is an $n$-ary $Endo$-prime hyperideal associated with $\theta$, by Proposition \ref{bagheri} (3). It follows that $u \notin E$ as $\theta\big(k(u^{(r)},1_H^{(n-r)}) \big)\notin E$ with  $r \leq n $ and $\theta \big(k_{(l)}(u^{(r)})\big) \notin E$ with $r>n, \ r=l(n-1)+1$. This is impossible. Thus, we conclude that $\Upsilon_{\theta}(H)$ is the intersection of all    $n$-ary $Endo$-prime hyperideals of $H$  associated with  $\theta$.
\end{proof}
For any given hyperideal $E$ of $H$, we define the $\theta$-radical of $E$ as following
\[rad_{\theta}(E)=\biggm{\{} u \in H \ \vert \   \biggm{\{} 
\begin{array}{lr}
\theta \big(k(u^{(r)},1_H^{(n-r)})\big) \in E,& r \leq n\\
\theta \big(k_{(l)}(u^{(r)})\big) \in E,  & r>n, \ r=l(n-1)+1
\end{array}
\biggm{\}}  \biggm{\}}.\]
It is clear that $\Upsilon_{\theta}(H)=rad_{\theta}(0)$.
\begin{corollary} \label{6}
Let $H$ be an $n$-ary hyperintegral domain.
Then $\theta$-radical of $E$ is the intersection of all    $n$-ary $Endo$-prime hyperideals of $H$  associated with  $\theta$ which contains $E$.
\end{corollary}
\begin{proof}
It follows by applying Theorem \ref{5} to the quotient Krasner $(m, n)$-hyperring $H/E$.
\end{proof}
Let $E$ be a proper hyperideal of $H$ and $\theta : H \longrightarrow H$ be an   endomorphism. We say that $E$ is an $n$-ary strongly $Endo$-prime hyperideal associated with $\theta$, if for hyperideals $U_1^n$ of $H$, $k(U_1^n) \subseteq E$ implies that $U_i \in E$ or $\theta\big(k(U_ 1^{i-1},1_H,U_{i+1}^n)\big) \subseteq E$ for some $i \in \{1,\ldots,n\}$. By the deﬁnition, it is concluded that every $n$-ary strongly $Endo$-prime hyperideal associated with $\theta$ is an $n$-ary  $Endo$-prime hyperideal associated with $\theta$. Moreover, the Krasner $(m,n)$-hyperring $H$ refers to an $n$-ary $\theta$-hyperintegral domain if $0$ is an $n$-ary $Endo$-prime hyperideal associated to $\theta$. Also, it is easy to check that $\theta_E: H/E \longrightarrow H/E$, defined by $ h(u,E,0^{(m-2)}) \mapsto h(\theta(u),E,0^{(m-2)})$, is a well-defined endomorphism on $H/E$.
\begin{theorem}\label{7}
Let $\theta : H \longrightarrow H$ be an   endomorphism and $E$ be a proper hyperideal of $H$. Then Then, the following statements hold:
\begin{itemize} 
\item[\rm{(1)}]~ If $E$ is an $n$-ary strongly $Endo$-prime hyperideal associated with $\theta$, then  $H/E$ is an $n$-ary $\theta_E$-hyperintegral domain.
\item[\rm{(2)}]~ If $H/E$ is an $n$-ary $\theta_E$-hyperintegral domain, then $E$ is an $n$-ary   $Endo$-prime hyperideal associated with $\theta$.
\end{itemize}
\end{theorem}
\begin{proof}
(1)  Assume that $E$ is an $n$-ary strongly $Endo$-prime hyperideal associated with $\theta$ and $h(u_{11}^{1(j-1)},E,u_{1(j+1)}^{1m}),\ldots, h(u_{n1}^{n(j-1)},E,u_{n(i+1)}^{nm}) \in H/E$ for $u_{11}^{1m},\ldots,u_{n1}^{nm} \in H$ such that $k \big(h(u_{11}^{1(j-1)},E,u_{1(j+1)}^{1m}),\ldots, h(u_{n1}^{n(j-1)},E,u_{n(j+1)}^{nm}) \big)=E$. This implies that $h \big((k(u_{11}^{n1}),\ldots,k(u_{1(j-1)}^{n(j-1)}),E,k(u_{1(j+1)}^{n(j+1)}),\ldots,k(u_{1m}^{nm}) \big)=E$. Therefore we conclude that 
$h \big((k(u_{11}^{n1}),\ldots,k(u_{1(j-1)}^{n(j-1)}),0,k(u_{1(j+1)}^{n(j+1)}),\ldots,k(u_{1m}^{nm}) \big) \subseteq E$ which means  $k(\big(h(u_{11}^{1(j-1)},0,u_{1(j+1)}^{1m}),\ldots, h(u_{n1}^{n(j-1)},0,u_{n(j+1)}^{nm}) \big) \subseteq E$. Since $E$ is an $n$-ary strongly $Endo$-prime hyperideal associated with $\theta$, we get $h(u_{i1}^{i(j-1)},0,u_{i(j+1)}^{im}) \subseteq E$ or

$\theta \big( k\big(h(u_{11}^{1(j-1)},0,u_{1(j+1)}^{1m}),\ldots,h(u_{(i-1)1}^{(i-1)(j-1)},0,u_{(i-1)(j+1)}^{(i-1)m}),h(1_H,0^{(m-1)}),$

$ \hspace{1cm} h(u_{(i+1)1}^{(i+1)(j-1)},0,u_{(i+1)(j+1)}^{(i+1)m}),\ldots, h(u_{n1}^{n(j-1)},0,u_{n(j+1)}^{nm}) \big) \subseteq E$\\
 for some $i \in \{1,\ldots,n\}$.  This implies that  $ h(u_{i1}^{i(j-1)},E,u_{i(j+1)}^{im}) = E$ or 
 
 $\theta_E \big( k (\big(h(u_{11}^{1(j-1)},E,u_{1(j+1)}^{1m}),\ldots,h(u_{(i-1)1}^{(i-1)(j-1)},E,u_{(i-1)(j+1)}^{(i-1)m}),h(1_H,E,0^{(m-2)}),$
 
 $ \hspace{1cm} h(u_{(i+1)1}^{(i+1)(j-1)},E,u_{(i+1)(j+1)}^{(i+1)m}),\ldots, h(u_{n1}^{n(j-1)},E,u_{n(j+1)}^{nm}) \big) \big)$

 $=\theta_E \big( h \big( k(\big(h(u_{11}^{1(j-1)},0,u_{1(j+1)}^{1m}),\ldots,h(u_{(i-1)1}^{(i-1)(j-1)},0,u_{(i-1)(j+1)}^{(i-1)m}),h(1_H,0^{(m-1)}),$
 
 $ \hspace{1cm} h(u_{(i+1)1}^{(i+1)(j-1)},0,u_{(i+1)(j+1)}^{(i+1)m}),\ldots, h(u_{n1}^{n(j-1)},0,u_{n(j+1)}^{nm}) \big),E,k(0^{(n)})^{(m-2)} \big)\big) \big)$
 
 $\hspace{0.3cm}=h \big( \theta \big( k\big(h(u_{11}^{1(j-1)},0,u_{1(j+1)}^{1m}),\ldots,h(u_{(i-1)1}^{(i-1)(j-1)},0,u_{(i-1)(j+1)}^{(i-1)m}),h(1_H,0^{(m-1)}),$
 
 $\hspace{1cm}  h(u_{(i+1)1}^{(i+1)(j-1)},0,u_{(i+1)(j+1)}^{(i+1)m}),\ldots, h(u_{n1}^{n(j-1)},0,u_{n(j+1)}^{nm}) \big),E,k(0^{(n)})^{(m-2)} \big)$
 
 $\hspace{0.3cm}= E. $\\
 Consequently, $H/E$ is an $n$-ary $\theta_E$-hyperintegral domain.
 
(2) Let $k(u_1^n) \in E$ for $u_1^n \in H$. Then $h(k(u_1^n),E,0^{(m-2)})=E$ by Lemma 3.4 in \cite{d1}. Therefore, we conclude that  $k(h(u_1,E,0^{(m-2)}),\ldots,h(u_n,E,0^{(m-2)}) \big)=h \big(k(u_1^n),k(0^{(n)}),\ldots,k(0^{(n)}),E,k(0^{(n)}),\ldots,k(0^{(n)})\big)=E$. Since $H/E$ is an $n$-ary $\theta_E$-hyperintegral domain, we obtain $h(u_i,E,0^{(m-2)}) =  E$ or

 $\hspace{1.8cm}h \big(\theta \big( k(u_1^{i-1},1_H,u_{i+1}^n) \big),E,0^{(m-2)} \big)$
 
 $\hspace{2cm} =\theta_E \big( h \big(k(u_1^{i-1},1_H,u_{i+1}^n),E,0^{(m-2)} \big) \big)$

 $\hspace{2cm} =\theta_E \big( k \big(h(u_1,E,0^{(m-2)}),\ldots,h(u_{i-1},E,0^{(m-2)}),h(1_H,E,0^{(m-2)}),$
 
 $\hspace{2.8cm}h(u_{i+1},E,0^{(m-2)}),\ldots,h(u_n,E,0^{(m-2)}) \big)\big)$
 
 $\hspace{2cm} =E.$

 This implies that $u_i \in E$ or $\theta \big( k(u_1^{i-1},1_H,u_{i+1}^n) \big) \in E$ and this is what we want to prove.
 \end{proof}


\begin{theorem}\label{10}
Let $\eta:H_1 \longrightarrow H_2$ be a homomorphism and  $\theta_i: H_i \longrightarrow H_i$ be  an endomorphism for every $i \in \{1,2\}$ such that $\theta_2 \circ \eta =\eta \circ \theta_1$. Then the following  holds:

\begin{itemize} 
\item[\rm{(1)}]~ If $E_2$ is an $n$-ary $Endo$-prime hyperideal of $H_2$ associated with $\theta_2$, then $\eta^{-1}(E_2)$  is an $n$-ary $Endo$-prime hyperideal of $H_1$ associated with $\theta_1$.
\item[\rm{(2)}]~If $\eta$ is an epimorphism 
 and  $E_1$ is an $n$-ary $Endo$-prime hyperideal of $H_1$ associated with $\theta_1$ such that $Ker \eta \subseteq E_1$, then $\eta(E_1)$ is an $n$-ary $Endo$-prime hyperideal of $H_2$ associated with $\theta_2$.
\end{itemize} 
\end{theorem}
\begin{proof}
(1) Assume that $k_1(u_1^n) \in \eta^{-1}(E_2)$ for $u_1^n \in H_1$. Then, we have $\eta \big(k_1(u_1^n)\big) \in E_2$ and so $k_2 \big(\eta(u_1),\ldots,\eta(u_n)\big) \in E_2$. Since $E_2$ is an $n$-ary $Endo$-prime hyperideal of $H_2$ associated with $\theta_2$, we obtain $\eta(u_i) \in E_2$ or $\eta  \big( \theta_1 \big(k_1(u_1^{i-1},1_{H_1},u_{i+1}^n) \big) \big)=\theta_2 \big( \eta \big(k_1(u_1^{i-1},1_{H_1},u_{i+1}^n) \big) \big)=  \theta_2 \big(k_2 \big( \eta(u_1),\ldots,\eta(u_{i-1}),1_{H_2},\eta(u_{i+1}), \ldots,\eta(u_n)\big) \big) \in E_2$ for some $i \in \{1,\ldots,n\}$. Therefore, $u_i \in \eta^{-1}(E_2)$ or $ \theta_1 \big(k_1(u_1^{i-1},1_{H_1},u_{i+1}^n) \big) \in \eta^{-1}(E_2)$. Consequently, $\eta^{-1}(E_2)$  is an $n$-ary $Endo$-prime hyperideal of $H_1$ associated with $\theta_1$.

(2) Suppose that $k_2(v_1^n) \in \eta(E_1)$ for $v_1^n \in H_2$. By the hypothesis, there exist $u_1^n \in H_1$ such that $\eta(u_i)=v_i$ for every $i \in \{1,\ldots,n\}$. Hence, we get $\eta \big( k_1(u_1^n)\big)=k_2 \big(\eta(u_1),\ldots,\eta(u_n) \big)=k_2(v_1^n) \in \eta(E_1)$. From $Ker \eta \subseteq E_1$, it follows that $k_1(u_1^n) \in E_1$. Since $E_1$ is an $n$-ary $Endo$-prime hyperideal of $H_1$ associated with $\theta_1$, we conclude that $u_i \in E_1$ or $\theta_1 \big(k_1(u_1^{i-1},1_{H_1},u_{i+1}^n) \big) \in E_1$ for some $i \in \{1,\ldots,n\}$. This implies that $v_i \in \eta(E_1)$ or $\theta_2 \big(k_2(v_1^{i-1},1_{H_2},v_{i+1}^n \big) \big)=\theta_2 \big( \eta       \big(k_1(u_1^{i-1},1_{H_1},u_{i+1}^n) \big) \big) =\eta \big(\theta_1 \big(k_1(u_1^{i-1},1_{H_1},u_{i+1}^n) \big) \big) \in \eta(E_1)$. Thus, $\eta(E_1)$ is an $n$-ary $Endo$-prime hyperideal of $H_2$ associated with $\theta_2$.
\end{proof}
\begin{corollary} \label{11}
Let $\theta : H \longrightarrow H$ be an   endomorphism and $E$ be a hyperideal of $H$.
\begin{itemize} 
\item[\rm{(1)}]~ Let $G$ be a subhyperring of $H$ containing $\theta(G)$. If $E$ is an $n$-ary $Endo$-prime hyperideal of $H$ associated with $\theta$, then $E \cap G$ is an $n$-ary $Endo$-prime hyperideal of $G$ associated with $\theta \vert_{G}$.
\item[\rm{(2)}]~ Assume that $E$ contains $Ker \theta$. Then, $E$ is an $n$-ary $Endo$-prime hyperideal of $H$ associated with $\theta$ if and only if $E/Ker \theta$ is  an $n$-ary $Endo$-prime hyperideal of $H/Ker \theta$ associated with $\theta_{Ker \theta}$.
\end{itemize} 
\end{corollary}
\begin{proof}
(1) It is sufficient to put $H_1=G$, $H_2=H$, $\theta_1=\theta \vert_{G}$, $\theta_2=\theta$ and consider $\eta$ as an inclusion homomorphism in Theorem \ref{10} (1).

(2) Assume that $E$ contains $Ker \theta$. The mapping $\pi :H \longrightarrow H/Ker \theta$, defined by $u \mapsto f(u,ker \theta,0^{(m-2)})$,  is an epimorphism by Theorem 3.2 in \cite{sorc1}. Now, put $H_1=H$, $H_2=H/Ker \theta$, $\theta_1=\theta$, $\theta_2=\theta_{Ker \theta}$ and $\eta=\pi$ in Theorem \ref{10}.
\end{proof}
A non-empty subset $S$  of $H$ is called $n$-ary multiplicative if $k(s_1^n) \in S$ for every $s_1^n \in S$ \cite{sorc1}. The notion of Krasner $(m,n)$-hyperring of fractions was   investigated in \cite{mah5}. Let $\theta : H \longrightarrow H$ be an   endomorphism. Define the mappings $\bar{\theta}: S^{-1}H \longrightarrow S^{-1}H$ by $\frac{u}{s} \mapsto \frac{\theta(u)}{\theta(s)}$  and $\bar{\bar{\theta}} :H \longrightarrow S^{-1}H$ by $u \mapsto \frac{u}{1}$ where $S \subseteq H$  is  an $n$-ary multiplicative set   containing $\theta(S)$ and $1_H$. A  routine check verifies that $\bar{\theta}$ and  $\bar{\bar{\theta}}$ are   homomorphisms of Krasner $(m,n)$-hyperring.
\begin{theorem} \label{8}
Let $\theta : H \longrightarrow H$ be an   endomorphism and $S \subseteq H$  be an $n$-ary multiplicative set containing $\theta(S)$ and $1_H$. Then, the following
 holds:
\begin{itemize} 
\item[\rm{(1)}]~
Let  $E$ be a hyperideal of $H$ with  $E \cap S=\varnothing$. If $E$ is an $n$-ary $Endo$-prime hyperideal of $H$ associated with $\theta$, then  $S^{-1}E$ an $n$-ary $Endo$-prime hyperideal of $S^{-1}H$ associated with $\bar{\theta}$.
\item[\rm{(2)}]~ If $F$ is   an $n$-ary $Endo$-prime hyperideal of $S^{-1}H$ associated with $\bar{\theta}$, then ${\bar{\bar{\theta}}}^{-1}(F)$ is an $n$-ary $Endo$-prime hyperideal of $ H$ associated with $\theta$.
\end{itemize} 
\end{theorem}
\begin{proof}
(1) Let $ K(\frac{u_1}{s_1},\ldots,\frac{u_n}{s_n}) \in S^{-1}E$ for $\frac{u_1}{s_1},\ldots,\frac{u_n}{s_n} \in S^{-1}H$.
Then $\frac{k(u_1^n)}{k(s_1^n)} \in S^{-1}E$ and so there exists $s \in S$ such that $ k(s,k(u_1^n),1_H^{(n-2)}) \in E$. Since $E$ is an $n$-ary $Endo$-prime hyperideal of $H$ associated with $\theta$ and  $k\big(k(s,u_1,1_H^{(n-2)}),u_2^n \big) \in E$, we get $k(s,u_1,1_H^{(n-2)}) \in E$ or  $\theta \big( k(1_H,u_2^n )\big) \in E$ or $u_i \in E$ for some $i \in \{2,\ldots,n\}$ or $\theta\big(k(s,u_1,1_H^{(n-2)}),u_2^{i-1},1_H,u_{i+1}^n \big) \in E$ . In the first possibility, we have $\frac{u_1}{s_1}=\frac{k(u_1,1_H^{(n-1)})}{k(s_1,1_H^{(n-2)})}=\frac{k(s,u_1,1_H^{(n-2)})}{k(s,s_1,1_H^{(n-2)})} \in S^{-1}E$. In the second possibility, we get $\bar{\theta}\big(K(\frac{1_H}{1_H},\frac{u_2}{s_2},\ldots,\frac{u_n}{s_n})\big)=\bar{\theta} \big(\frac{k(1_H,u_2^n )}{k(1_H,s_2^n )} \big)=\frac{\theta \big( k(1_H,u_2^n )\big)}{  \theta \big(k(1_H,s_2^n ) \big)} \in S^{-1}E$. In the third possibility, we obtain $\frac{u_i}{s_i} \in S^{-1}E$ for some $i \in \{1,\ldots,n\}$. In the fourth possibility, we have 
$\bar{\theta}\big(K(\frac{u_1}{s_1},\frac{u_2}{s_2},\ldots,\frac{u_{i-1}}{s_{i-1}},\frac{1_H}{1_H},\frac{u_{i+1}}{s_{i+1}},\ldots,\frac{u_n}{s_n})\big)=\bar{\theta} \big(\frac{k(u_1,u_2^{i-1},s,u_{i+1}^n )}{k(s_1,s_2^{i-1},s,s_{i+1}^n )} \big)=\frac{\theta\big(k(s,u_1,1_H^{(n-2)}),u_2^{i-1},1_H,u_{i+1}^n \big)}{  \theta \big(k(k(s,s_1),s_2^{i-1},1_H,s_{i+1}^n ) \big) } \in S^{-1}E$. Thus, $S^{-1}E$ an $n$-ary $Endo$-prime hyperideal of $S^{-1}H$ associated with $\bar{\theta}$.

(2) Since $\bar{\theta} \circ \bar{\bar{\theta}}=\bar{\bar{\theta}} \circ \theta$, we are done by Theorem \ref{10} (1).
\end{proof}
Here, we establish the following result.
\begin{corollary}
Let $\theta : H \longrightarrow H$ be an   endomorphism and $S \subseteq H$  be an $n$-ary multiplicative set containing $\theta(S)$ and $1_H$. Then there is a one-to-one correspondence between $n$-ary $Endo$-prime hyperideals $E$ of $H$ associated with $\theta$ disjoint with $S$ and $n$-ary $Endo$-prime hyperideals  of $S^{-1}H$ associated with $\bar{\theta}$.
\end{corollary}
Let $(H_1, h_1, k_1)$ and $(H_2, h_2, k_2)$ be two commutative Krasner $(m,n)$-hyperrings and, $1_{H_1}$ and $1_{H_2}$ be scalar identities of $H_1$ and $H_2$, respectively. Then the triple $(H_1 \times H_2, h _1 \times h_2 ,k_1 \times k_2 )$ is a Krasner $(m, n)$-hyperring such that $m$-ary hyperoperation 
$h _1 \times h_2 $ and $n$-ary operation $k_1 \times k_2$ are defined as follows:

$\hspace{1cm} h_1 \times h_2 \big((x_{1}, y_{1}),\ldots,(x_m,y_m) \big) = \{(x,y) \ \vert \ \ x \in h_1(x_1^m), y \in h_2(y_1^m) \}$

$\hspace{1cm} k_1 \times k_2 \big((u_1,v_1),\ldots,(u_n,v_n)) =(k_1(u_1^n),k_2(v_1^n) \big) $,\\
for $x_1^m,u_1^n \in H_1$ and $y_1^m,v_1^n \in H_2$ \cite{mah2}. 
\begin{theorem}\label{9}
Let $(H_1, h_1, k_1)$ and $(H_2, h_2, k_2)$ be two commutative Krasner $(m,n)$-hyperrings and, $\theta_1: H_1 \longrightarrow H_1$ and $\theta_2: H_2 \longrightarrow H_2$ be two endomorphisms. Then, the following statements are equivalent:
\begin{itemize} 
\item[\rm{(1)}]~ $E$ is an $n$-ary $Endo$-prime hyperideal of $H_1 \times H_2$ associated with $\theta_1 \times \theta_2$.
\item[\rm{(2)}]~ $E=E_1 \times H_2$ such that $E_1$ is an $n$-ary $Endo$-prime hyperideal of $H_1$ associated with $\theta_1$ or $E=H_1 \times E_2$ such that $E_2$ is an $n$-ary $Endo$-prime hyperideal of $H_2$ associated with $\theta_2$.
\end{itemize} 

\end{theorem}
\begin{proof}
(1) $\Longrightarrow$ (2) Assume that $E$ is an $n$-ary $Endo$-prime hyperideal of $H_1 \times H_2$ associated with $\theta_1 \times \theta_2$. Therefore, $E=E_1 \times E_2$ where $E_1$ and $E_2$ are hyperideals of $H_1$ and $H_2$, respectively. From $k_1 \times k_2 \big((1_{H_1},0),(0,1_{H_2})^{(n-1)} \big) \in E_1 \times E_2$, it follows that $(1_{H_1},0) \in E_1 \times E_2$ which means $E=H_1 \times E_2$ or $\theta_1 \times \theta_2 \big(k_1 \times k_2\big((1_{H_1},1_{H_2}),(0,1_{H_2})^{(n-1)} \big)\big)=(0,1_{H_2}) \in E_1 \times E_2$ which implies $E=E_1 \times H_2$. Let us assume that $E=H_1 \times E_2$. Suppose that $k_2(v_1^n) \in E_2$ for $v_1^n \in H_2$. Hence, we have $k_1 \times k_2 \big( (0,v_1),\ldots,(0,v_n) \big) \in H_1 \times E_2$. By the hypothesis, we conclude that $(0,v_i) \in H_1 \times E_2$ for some $i \in \{1,\ldots,n\}$ which means $v_i \in E_2$  or $\theta_1 \times \theta_2 \big( k_1 \times k_2 \big( (0,v_1),\ldots,(0,v_{i-1}),(1_{H_1},1_{H_2}), (0,v_{i+1}), \ldots,(0,v_n) \big)\big)=\big(0,\theta_2 \big(k_2(v_1^{i-1},1_{H_2},v_{i+1}^n) \big) \big) \in H_1 \times E_2$ which implies $\theta_2 \big(k_2(v_1^{i-1},1_{H_2},v_{i+1}^n) \big) \in E_2$. Thus, $E=H_1 \times E_2$ where $E_2$ is an $n$-ary $Endo$-prime hyperideal of $H_2$ associated with $\theta_2$.

(2) $\Longrightarrow$ (1)It is obvious.
\end{proof}
Now, we have the following result obtained by using mathematical induction on $t$ and previous theorem.
\begin{corollary}
Assume that $(H_i,h_i,k_i)$ is a Krasner $(m,n)$-hyperring and $\theta_i: H_i \longrightarrow H_i$ is an endomorphism for each $i \in \{1,\ldots,t\}$. Then, the following statements are equivalent:
\begin{itemize} 
\item[\rm{(1)}]~ $E$ is an $n$-ary $Endo$-prime hyperideal of $H_1 \times \cdots \times H_t$ associated with $\theta_1 \times \cdots \times \theta_t$. 
\item[\rm{(2)}]~ $E=E_1 \times \cdots \times E_t$ such that $E_j$ is an $n$-ary $Endo$-prime hyperideal of $H_j$ associated with $\theta_j$ for some $j \in \{1,\ldots,t\}$ and $E_i=H_i$ for every $i \neq j$.
\end{itemize}
\end{corollary}
\section{an expansion of the $Endo$-hyperideals}
In this section, we generalize the notion of $n$-ary $Endo$-prime hyperideals in a Krasner $(m,n)$-hyperring introducing and studying the concept of $n$-ary $Endo$-primary hyperideals.
\begin{definition} 
Let $E$ be a proper hyperideal of $H$ and $\theta : H \longrightarrow H$ be an   endomorphism.  $E$ refers to an $n$-ary $Endo$-primary  hyperideal associated with $\theta$, if for all $u_1^n \in H$, $k(u_1^n) \in E$ implies that $u_i \in E$ or $\theta\big(k(u_ 1^{i-1},1_H,u_{i+1}^n)\big) \in rad(E)$ for some $i \in \{1,\ldots,n\}$. 
\end{definition}
Although every $n$-ary $Endo$-prime  hyperideal associated with $\theta$ is an $n$-ary $Endo$-primary  hyperideal associated with $\theta$, the following verifies that the converse of the explanation may not be generally
true. 
\begin{example} 
Consider the Krasner $(2,3)$-hyperring  $(H=[0,1],\oplus, \cdot)$ where  the 2-ary hyperoperation $"\oplus"$ defined by 

\[ 
u \oplus v=
\begin{cases}
$[0,u]$, & \text{if $u =v$}\\
\{\max\{u,v\}\}, & \text{if $u \neq v$}
\end{cases} \]

and  $"\cdot"$ is the usual multiplication on real numbers. In  this hyperring,  the hyperideal $E=[0,0.5]$ is a  $3$-ary $Endo$-primary hyperideal associated with $\theta$ where $\theta$ is the inclusion homomorphism. However, $E$ is not a  $3$-ary $Endo$-primary hyperideal associated with $\theta$ since $0.9 \cdot 0.9 \cdot 0.6 \in E$ but $0.9, 0.6 \notin E$ and $0.9 \cdot 0.9, 0.9 \cdot 0.6 \notin E$. 
\end{example}
 Theorem 4.28 in \cite{sorc1}  verifies that the radical of an $n$-ary  primary hyperideal of $H$ is an $n$-ary prime hyperideal. Now, the following theorem shows that the radical of an $n$-ary $Endo$-primary hyperideal associated with $\theta$ is an $n$-ary $Endo$-prime hyperideal associated with $\theta$.
\begin{theorem}  
Let  $E$ be an  $n$-ary $Endo$-primary hyperideal associated with $\theta$ where $\theta : H \longrightarrow H$ is an   endomorphism. Then   $rad(E)$ is an $n$-ary $Endo$-prime hyperideal associated with $\theta$.
\end{theorem} 
\begin{proof} 
Let $k(u_1^n) \in rad(E)$ for $u_1^n \in H$. Then, there exists $r \in \mathbb{N}$ such that if $r \leq n$, then $k\big(k(u_1^n)^{(r)},1_H^{n-r}\big) \in E$. Hence, $k\big(u_i^{(r)},k(u_1^{i-1},1_H,u_{i+1}^n)^{(r)},1_H^{(n-2r)} \big) \in E$ and so $k\big(k(u_i^{(r)},1_H^{(n-r)}),k(k(u_1^{i-1},1_H,u_{i+1}^n)^{(r)},1_H^{(n-r)}),1_H^{(n-2)} \big) \in E$. 
Since $E$ is an  $n$-ary $Endo$-primary hyperideal associated with $\theta$,   $ k(u_i^{(r)},1_H^{(n-r)}) \in E$ or $k \big(\theta\big(k(u_1^{i-1},1_H,u_{i+1}^n))^{(r)},1_H^{(n-r)}\big) \big)=\theta \big(k\big(k(u_1^{i-1},1_H,u_{i+1}^n)^{(r)},1_H^{(n-r)}\big) \big) \in rad(E)$. This means that $u_i \in rad(E)$ or $\theta\big(k(u_ 1^{i-1},1_H,u_{i+1}^n)\big) \in rad(E)$. Thus, we conclude that $rad(E)$ is an $n$-ary $Endo$-prime hyperideal associated with $\theta$.
Now, let $r=l(n-1)+1$. So, we get $k_{(l)}(k(u_1^n)^{(r)}) \in E$. In this case, by a similar argument we conclude that  $rad(E)$ is an $n$-ary $Endo$-prime hyperideal associated with $\theta$.
\end{proof} 
\begin{proposition} 
Let $E$ be a proper hyperideal of $H$, $\theta : H \longrightarrow H$  an   endomorphism and $u \in H$. If $E$ is   an $n$-ary $Endo$-primary hyperideal associated with $\theta$, then the following holds:
\begin{itemize} 
\item[\rm{(1)}]~$\theta(E) \subseteq rad(E)$.
\item[\rm{(2)}]~$k \big(u^ {(r)} , 1_H^{(n-r)} ) \big) \in E$ with $r \leq n$, or $  k_{(l)} (u^ {(r)} )  \in E$ with $r = l(n-1) + 1$ implies that $u \in E$ or $\theta (u) \in rad(E)$.
\end{itemize} 
\end{proposition}
\begin{proof}
(1) Assume that  $u \in E$. So, we get $k(u,1_H^{(n-1)}) \in E$.  Since $E$ is   an $n$-ary $Endo$-primary hyperideal associated with $\theta$ and $1_H \notin E$, we have    $\theta(u)=\theta \big(k(u,1_H^{(n-1)}) \big) \in rad(E)$ which means $\theta(E) \subseteq rad(E)$.

(2) Let  $k \big(u^ {(r)} , 1_H^{(n-r)} ) \big) \in E$. Then,  we have $k(u,k(u^{(r-1)},1_H^{(n-r+1)}),1_H^{(n-2)}) \in E$. Since $E$ is   an $n$-ary $Endo$-primary hyperideal associated with $\theta$, we conclude that  $k(u^{(r-1)},1_H^{(n-r+1)}) \in E$ or $\theta(u) \in rad(E)$. In the first possibility, we get $k(u,k(u^{(r-2)},1_H^{(n-r+2)}),1_H^{(n-2)}) \in E$ which means $k(u^{(r-2)},1_H^{(n-r+2)}) \in E$ or $\theta(u) \in rad(E)$. Continuing in this manner, we have $u \in E$ or $\theta(u) \in rad(E)$, as needed. If $  k_{(l)} (u^ {(r)} )  \in E$ with $r = l(n-1) + 1$, then by a similar argument we get the result that $u \in E$ or $\theta(u) \in rad(E)$.
\end{proof}
Let $\theta : H \longrightarrow H$ be an   endomorphism. A proper hyperideal $M$ of $H$ refers to a $\theta$-maximal hyperideal if $M \subseteq E$ for some hyperideal $E$ of $H$, then $\theta(E) \subseteq M$ or $E=H$.
\begin{proposition}
Let $\theta : H \longrightarrow H$ be an   endomorphism. Then $0$ is a $\theta$-maximal hyperideal of $H$ if and only if $Max(H)=\{Ker \theta\}$.
\end{proposition}
\begin{proof}
$ \Longrightarrow$ Assume that $0$ is a $\theta$-maximal hyperideal of $H$. Take any $u \in H$ such that $u \notin Ker \theta$. This means that $\theta(u) \neq 0$ and so $\theta( \langle u \rangle) \nsubseteq 0$. By the hypothesis, we conclude that $\langle u \rangle=H$. It follows that $1_H=k(v,u,1_H^{(n-2)})$ for some $v \in H$ and so $u$ is invertible. Therefore, $Max(H)=\{Ker \theta\}$.

$\Longleftarrow$ Let $Max(H)=\{Ker \theta\}$ and $\theta(E) \neq 0$ for some hyperideal $E$ of $H$. It follows that $\theta(u) \neq 0$ for some $u \in E$ which means $u \notin Ker \theta$. This means that $u$ is invertible by the assumption and so $E=H$. Thus, $0$ is a $\theta$-maximal hyperideal of $H$.
\end{proof}
\begin{theorem}
Let $\theta : H \longrightarrow H$ be an   endomorphism. If $E_1,\ldots,E_t$ are $n$-ary $Endo$-primary hyperideals of $H$ associated with $\theta$ such that $rad(E_j)=rad(E_l)$ for all $j,l \in \{1,\ldots,t\}$, then $\cap_{j=1}^tE_j$ is an $n$-ary $Endo$-primary hyperideals of $H$ associated with $\theta$.
\end{theorem}
\begin{proof}
Since $rad(E_j)=rad(E_l)$ for all $j,l \in \{1,\ldots,t\}$, we assume that $rad(E_j)=Q$ for all $j \in \{1,\ldots,t\}$. Let $k(u_1^n) \in \cap_{j=1}^tE_j$ for $u_1^n \in H$ such that $u_i \notin \cap_{j=1}^tE_j$ for some $i  \in \{1,\ldots,n\}$. Then there exists $i_0 \in \{1,\ldots,t\}$ such that $u_i \notin E_{i_0}$. Since $ E_{i_0}$ is an $n$-ary $Endo$-primary hyperideal of $H$ associated with $\theta$ and $u_i \notin E_{i_0}$, we conclude that $\theta \big( k(u_1^{i-1},1_H,u_{i+1}^n) \big) \in rad(E_{i_0})=Q$. As $rad(\cap_{j=1}^tE_j)=\cap_{j=1}^t rad(E_j)=Q$, we get $\theta \big( k(u_1^{i-1},1_H,u_{i+1}^n) \big) \in rad(\cap_{j=1}^tE_j)$. Thus, $\cap_{j=1}^tE_i$ is an $n$-ary $Endo$-primary hyperideals of $H$ associated with $\theta$.
\end{proof}
\begin{corollary}
If $E_1,\ldots,E_t$ are $n$-ary  primary hyperideals of $H$  such that $rad(E_j)=rad(E_l)$ for all $j,l \in \{1,\ldots,t\}$, then $\cap_{j=1}^tE_j$ is an $n$-ary  primary hyperideals of $H$.
\end{corollary}
\begin{proof}
It is sufficient to consider $\theta$ as an inclusion homomorphism.
\end{proof}
\begin{proposition}
Let $\theta : H \longrightarrow H$ be an   endomorphism and $E$ be an $n$-ary $Endo$-primary hyperideal associated with $\theta$  such that $rad(E)=Q$. Then the following holds:
\begin{itemize} 
\item[\rm{(1)}]~ $(E:u)=H$ for all $u \in E$.
\item[\rm{(2)}]~$(E:u)=E$ for all $u \in H$ such that $\theta(u) \notin Q$.
\end{itemize} 
\end{proposition}
\begin{proof}
(1) It is obvious.

(2) Let $\theta(u) \notin Q$ for $u \in H$. The inclusion $E \subseteq (E:u)$ always holds. Assume that $v \in (E:u)$. Then $k(v,u,1_H^{(n-2)}) \in E$. Since $E$ be an $n$-ary $Endo$-primary hyperideal associated with $\theta$ and $\theta(u) \notin Q=rad(E)$, we conclude that $v \in E$ and so $(E:u) \subseteq E$. Thus, $(E:u)=E$.
\end{proof}
We end this section with following theorem.
\begin{theorem}
Let $\theta : H \longrightarrow H$ be an   endomorphism. Then the following holds:
\begin{itemize} 
\item[\rm{(1)}]~Every $\theta$-maximal hyperideal of $H$ is an $n$-ary $Endo$-prime hyperideal associated with $\theta$.
\item[\rm{(2)}]~ If the radical of a hyperideal $E$ of $H$ is a $\theta$-maximal hyperideal, then $E$ is  an  $n$-ary $Endo$-primary hyperideal associated with $\theta$.
\end{itemize}
\end{theorem}
\begin{proof}
(1) Let $M$ be a $\theta$-maximal hyperideal of $H$ and $k(u_1^n) \in M$ for $u_1^n \in H$ but $\theta \big( k(u_1^{i-1},1_H,u_{i+1}^n) \big) \notin M$. By Lemma 3.4 in \cite{d1}, $ h \big(M, \langle k(u_1^{i-1},1_H,u_{i+1}^n)  \rangle,0^{(m-2)} \big)$ is a hyperideal of $H$. Clearly, $M \subseteq h \big(M, \langle k(u_1^{i-1},1_H,u_{i+1}^n)  \rangle,0^{(m-2)} \big)$. Since $M$ is  a $\theta$-maximal hyperideal of $H$ and $\theta \big( h \big(M, \langle k(u_1^{i-1},1_H,u_{i+1}^n)  \rangle,0^{(m-2)} \big) \big) \nsubseteq  M$, we get $h \big(M, \langle k(u_1^{i-1},1_H,u_{i+1}^n)  \rangle,0^{(m-2)} \big)=H$. Then there exists $m \in M$ and $a \in H$ auch that  $1_H \in h \big(m,k(u_1^{i-1},a,u_{i+1}^n),0^{(m-2)} \big)$ which implies\\ 

$\hspace{1.5cm} k(u_i,1_H^{(n-2)}) \in  k(u_i,h \big(m,k(u_1^{i-1},a ,u_{i+1}^n),0^{(m-2)} \big),1_H^{(n-2)} \big)$

$\hspace{3.4cm}=h \big( k(u_i,m,1^{(n-2)}), k(u_i,k(u_1^{i-1},a ,u_{i+1}^n),1_H^{(n-2)}),0^{(m-2)} \big)$

$\hspace{3.4cm}=h \big( k(u_i,m,1^{(n-2)}), k(a,k(u_1^n),1_H^{(n-2)}),0^{(m-2)} \big)$

$\hspace{3.4cm} \in M.$\\

Hence, $M$ is an $n$-ary $Endo$-prime hyperideal associated with $\theta$.

(2)  Suppose that $E$ is a hyperideal of $H$ such that $rad(E)$ is a $\theta$-maximal hyperideal. Let $k(u_1^n) \in E$ for $u_1^n \in H$ such that  $\theta  \big( k(u_1^{i-1},1_H,u_{i+1}^n) \big) \notin rad(E)$. Then, we conclude that  $h \big( rad(E),\langle k(u_1^{i-1},1_H,u_{i+1}^n) \rangle, 0^{(m-2)} \big)=H$ since  $rad(E)$ is a $\theta$-maximal hyperideal, $rad(E) \subseteq h \big( rad(E),\langle k(u_1^{i-1},1_H,u_{i+1}^n) \rangle, 0^{(m-2)} \big)$   and $\theta \big( h \big( rad(E),\langle k(u_1^{i-1},1_H,u_{i+1}^n) \rangle, 0^{(m-2)} \big)  \big) \nsubseteq rad(E)$. Now, let us to assume that  $h \big( E,\langle k(u_1^{i-1},1_H,u_{i+1}^n) \rangle, 0^{(m-2)} \big) \neq H$. So,  $E \subseteq h \big( E,\langle k(u_1^{i-1},1_H,u_{i+1}^n) \rangle, 0^{(m-2)} \big) \subseteq Q$ for some $n$-ary prime hyperideal $Q$ of $H$.  This implies that  $rad(E) \subseteq Q$ and $\langle k(u_1^{i-1},1_H,u_{i+1}^n) \rangle \subseteq h \big( E,\langle k(u_1^{i-1},1_H,u_{i+1}^n) \rangle, 0^{(m-2)} \big) \subseteq Q$. Therefore, we obtain $H= h \big( rad(E),\langle k(u_1^{i-1},1_H,u_{i+1}^n) \rangle, 0^{(m-2)} \big) \subseteq Q$, a contradiction. Hence, we get $h \big( E,\langle k(u_1^{i-1},1_H,u_{i+1}^n) \rangle, 0^{(m-2)} \big) = H$ and so $1_H \in h \big(u,k(u_1^{i-1},a,u_{i+1}^n),0^{(m-2)} \big)$ for some $u \in E$ and $a \in H$. Then, \\

$\hspace{1.5cm} k(u_i,1_H^{(n-2)}) \in  k(u_i,h \big(u,k(u_1^{i-1},a ,u_{i+1}^n),0^{(m-2)} \big),1_H^{(n-2)} \big)$

$\hspace{3.4cm}=h \big( k(u_i,u,1^{(n-2)}), k(u_i,k(u_1^{i-1},a ,u_{i+1}^n),1_H^{(n-2)}),0^{(m-2)} \big)$

$\hspace{3.4cm}=h \big( k(u_i,u,1^{(n-2)}), k(a,k(u_1^n),1_H^{(n-2)}),0^{(m-2)} \big)$

$\hspace{3.4cm} \in E.$\\

Consequently, $E$ is  an  $n$-ary $Endo$-primary hyperideal associated with $\theta$.
\end{proof}
\section{conclusion}
The current paper has  introduced  the notion of $n$-ary $Endo$-prime hyperideals using an endomorphism $\theta$ and investigated their properties. Although the class of  $n$-ary $Endo$-prime hyperideals is a generalization of $n$-ary prime hyperideals, we gave an example showed an $n$-ary $Endo$-prime hyperideal may not be an $n$-ary prime hyperideal.  
We concluded that a hyperideal  $E$ with $\theta(E) \subseteq E$ may not be an $n$-ary $Endo$-prime hyperideal, however, the inclusion  $\theta(E) \subseteq E$ always hold for every $n$-ary $Endo$-prime hyperideal $E$.
We presented several characterizations of this new concept under certain conditions. The findings discussed in this paper enhance the comprehension of the structure of Krasner $(m,n)$-hyperrings with  respect to an endomorphism $\theta$. We examined many behaviors of $Endo$-prime hyperideals in specific cases.  Furthermore, we presented the concept of $n$-ary $Endo$-primary  hyperideals as an expansion of $n$-ary $Endo$-prime hyperideals. We showed that an $n$-ary $Endo$-primary hyperideals may not be an $n$-ary $Endo$-prime hyperideal. We ontained that the radical of an $n$-ary $Endo$-primary hyperideal is $n$-ary $Endo$-prime hyperideal. As a new research subject, we suggest the notion of  $Endo$-prime subhypemodules. 



\begin{thebibliography}{99}
\bibitem{Akray} 
I. Akray, H.M. Mohammed-Salih, $\alpha$-prime ideals, {\it Journal of Mathematical Extension}, {\bf 16}(1) (2022) (2)1-11.

\bibitem{sorc1} 
R. Ameri, M. Norouzi, Prime and primary hyperideals in Krasner $(m,n)$-hyperrings, {\it European Journal Of Combinatorics}, (2013) 379-390.


\bibitem{mah2} 
M. Anbarloei, $n$-ary 2-absorbing and 2-absorbing primary hyperideals in Krasner $(m,n)$-hyperrings, {\it Matematicki Vesnik,} {\bf 71} (3) (2019) 250-262. 




\bibitem{mah5}
M. Anbarloei, Krasner $(m,n)$-hyperring of fractions, {\it Jordan Journal of Mathematics and Statistic,} {\bf 16}(1) (2023) 165-185.

\bibitem{mah6}
M. Anbarloei, $N$-hyperideals and related extensions, {\it Filomat}, {\bf 38}(8) (2024) 2755-2772.


\bibitem{asadi}
A. Asadi, R. Ameri, Direct limit of Krasner (m,n)-hyperrings, {\it Journal of Sciences}, {\bf 31} (1) (2020) 75-83.








\bibitem{car}
B. Davvaz, Fuzzy Krasner $(m, n)$-hyperrings, {\it Comput. Math. with Appl.,} {\bf 59} (2010) 3879-3891.

\bibitem{www}
B. Davvaz, G. Ulucak, U. Tekir, Weakly $(k, n)$-absorbing (primary) hyperideals of a Krasner $(m,n)$-hyperring, {\it Hacettepe Journal of
Mathematics and Statistics}, {\bf 52}(5) (2023) 1229-1238.






\bibitem{rev2}
K. Hila, K. Naka, B. Davvaz, On $(k,n)$-absorbing
hyperideals in Krasner $(m,n)$-hyperrings, {\it Quarterly Journal of
Mathematics}, {\bf 69}
(2018) 1035-1046.


\bibitem{krasner}
M. Krasner, A class of hyperrings and hyperfields,  {\it Int. J. Math. Math. Sci.,} {\bf 2} (1983)  307– 312.








\bibitem{ma}
X. Ma, J. Zhan, B. Davvaz, Applications of rough soft sets to Krasner $(m,n)$-hyperrings and corresponding decision making methods, {\it Filomat}, {\bf 32} (2018) 6599-6614.


\bibitem{s1} 
F. Marty, Sur une generalization de la notion de groupe, {\it $8^{th}$ Congress Math. Scandenaves, Stockholm,} (1934) 45-49.


\bibitem{d1}
S. Mirvakili, B. Davvaz, Relations on Krasner $(m,n)$-hyperrings, {\it European J. Combin.,} {\bf 31}(2010) 790-802.


 \bibitem{Mahdou}
 N. Mahdou, S. Moussaoui, U.  Tekir and S.  Koc, A generalization of primary ideals of commutative rings, {\it Asian-European Journal of Mathematics} (2025). DOI: 10.1142/S1793557125500688




\bibitem{Najjar }
A.E. Najjar, A.S. Mohammed, $Endo$-prime ideals of commutative rings, {\it Commun. Korean Math. Soc.} {\bf 40}(3) (2025) 501–517.


\bibitem{nour}
M. Norouzi, R.Ameri, V. Leoreanu-Fotea, Normal hyperideals in Krasner $(m,n)$-hyperrings, {\it An. St. Univ. Ovidius Constanta} {\bf 26} (3) (2018) 197-211.


\bibitem{rev1}
S. Ostadhadi-Dehkordi, B. Davvaz, A Note on
Isomorphism Theorems of Krasner $(m,n)$- hyperrings, {\it Arabian Journal of
Mathematics,} {\bf 5} (2016) 103-115.




\bibitem{s4}
T. Vougiouklis, Hyperstructures and their representations, {\it Hadronic Press Inc., Florida, }(1994).







\end{thebibliography}
\end{document}